\documentclass[11pt,reqno,tbtags]{article}
\usepackage{amssymb,amsthm}
\usepackage[marginratio=1:1,height=584pt,width=488pt,tmargin=117pt]{geometry}
\usepackage{cite}
\usepackage{amsmath,amssymb,amsfonts,tikz}
\usepackage{enumerate,graphicx}

\numberwithin{equation}{section}

\newtheorem{maintheorem}{Theorem}

\newtheorem{theorem}{Theorem}[section]
\newtheorem*{theorem*}{Theorem}
\newtheorem{lemma}[theorem]{Lemma}
\newtheorem{claim}[theorem]{Claim}
\newtheorem{proposition}[theorem]{Proposition}
\newtheorem{corollary}[theorem]{Corollary}

\newtheorem{definition}[theorem]{Definition}
\newtheorem*{question*}{Question}

\theoremstyle{definition}{
\newtheorem{example}[theorem]{Example}

\newtheorem{remark}[theorem]{Remark}
\newtheorem*{remark*}{Remark}

}

\newcommand{\E}{\mathbb{E}}
\renewcommand{\P}{\mathbb{P}}
  
\renewcommand{\epsilon}{\varepsilon}

\newcommand{\cA}{\mathcal A}
\newcommand{\cB}{\mathcal B}
\newcommand{\cC}{\mathcal C}

\newcommand{\cG}{{\mathcal{G}}}

\newcommand{\GC}{{\mathcal{C}_1}} 

\newcommand{\TC}[1][\mathcal{C}_1]{#1^{(2)}} 

\newcommand{\Bin}{\operatorname{Bin}}

\DeclareMathOperator{\diam}{diam}
\DeclareMathOperator{\dist}{dist}

\DeclareMathOperator{\xor}{\triangle}
\newcommand{\red}{{\textsc{r}}}

\title{\vspace{-0.9cm} Asymptotics in percolation on high-girth expanders}

\author{{Michael Krivelevich\thanks{School of Mathematical
Sciences, Raymond and Beverly Sackler Faculty of Exact Sciences, Tel
Aviv University, Tel~Aviv 6997801, Israel. Email: {\ttfamily
krivelev@post.tau.ac.il}. Research supported in part by BSF grant 2014361 and ISF grant 1261/17.}}
 \and {Eyal Lubetzky\thanks{Courant Institute of Mathematical Sciences, New York University, 251 Mercer Street, New York, NY 10012, USA. Email: {\tt eyal@courant.nyu.edu}.
 Research supported in part by NSF grant DMS-1513403 and BSF grant 2014361.}}
\and {Benny Sudakov\thanks{
Department of Mathematics, ETH, Ramistrasse 101, 8092 Zurich, Switzerland. 
Email: {\tt benjamin.sudakov@math.ethz.ch}.
Research supported in part by SNSF grant 200021-175573.}}
}
\date{}

\begin{document}

\maketitle

\begin{abstract}
We consider supercritical bond percolation on a family of high-girth $d$-regular expanders. Alon, Benjamini and Stacey (2004) established that its critical probability for the appearance of a linear-sized (``giant'') component is $p_c=1/(d-1)$. Our main result recovers the sharp asymptotics of the size and degree distribution of the vertices in the giant and its 2-core at any $p>p_c$. It was further shown in~\cite{ABS04} that the second largest component, at any $0<p<1$, has size at most $n^{\omega}$ for some $\omega<1$. We show that, unlike the situation in the classical Erd\H{o}s--R\'enyi random graph, the second largest component in bond percolation on a regular expander, even with an arbitrarily large girth, can have size $n^{\omega'}$ for $\omega'$ arbitrarily close to $1$. Moreover, as a by-product of that construction, we answer negatively a question of Benjamini~(2013) on the relation between the diameter of a component in percolation on expanders and the existence of a giant component. Finally, we establish other typical features of the giant component, e.g., the existence of a linear~path.
\end{abstract}

\section{Introduction}
A graph $\cG$ is called a $(b,d)$-expander if its maximum degree is $d$, and for every subset  $S\subset V(\cG)$ of  at most $|V(\cG)|/2$ vertices there are at least $b |S|$ edges in the cut between $S$ and $V(\cG)\setminus S$.
 For $0<p<1$ and a given graph $\cG$, let $\cG_p$ denote the corresponding bond percolation on $\cG$, i.e., the distribution over spanning subgraphs of $\cG$ where each edge is present, independently, with probability~$p$.
Our primary focus  will be $\cG_p$ for a ($d$-)regular $(b,d$)-expander $\cG$ whose girth  (length of the shortest cycle) is large.

The pioneering paper of Alon, Benjamini and Stacey~\cite{ABS04} showed that if $\cG$ is an expander and $p>0$ is fixed, there is at most a single linear component in $\cG_p$ with high probability (w.h.p.). Moreover, when $\cG$ is $d$-regular and its girth tends to $\infty$, the authors of~\cite{ABS04} identified the critical percolation probability to be $p_c=1/(d-1)$, showing that, for every fixed $p>p_c$, w.h.p.\ there exists a linear (``giant'') component in $\cG_p$  (see the formal statement below). The latter result was thereafter extended by Peres et al.~\cite{BNP11} to any family of sparse graphs that converge as $n\to\infty$ in the Benjamini--Schramm sense.

More formally, write $\cC_i(G)$ (for $i=1,2,\ldots$) for the connected components of $G$ in decreasing order of their sizes. (By a slight abuse of notation, we also let $\cC_v(G)$ (for $v\in V(G)$)  denote the connected component of $v$ in $G$.)
It was shown in~\cite{ABS04} that if $\cG$ is a regular $(b,d)$-expander on $n$ vertices (more precisely, a sequence of such graphs) with girth tending to $\infty$ with $n$, then for every fixed $p>1/(d-1)$,
\begin{align*} \mbox{for \emph{some} fixed $c=c(p,d)>0$\,:\; }&\lim_{n\to\infty}\P\left(|\cC_1(\cG_p)|>cn\right) = 1\,,	\\
\noalign{\noindent whereas for every fixed $p<1/(d-1)$,}
 \mbox{for \emph{every} fixed $c>0$\,:\; }&\lim_{n\to\infty}\P\left(|\cC_1(\cG_p)|>cn\right) = 0\,.
\end{align*}
That the lack of a linear component w.h.p.\ at 
$p<1/(d-1)$ extends also to $p=1/(d-1)$  follows from the work of Nachmias and Peres~\cite{NP10} (who proved, more generally, that $\P(|\cC_1| \leq A n^{2/3})\to 0$ as $A\to\infty$ in percolation with parameter $p=1/(\Delta-1)$ on any family of graphs with maximal degree~$\Delta$).

Pittel~\cite{Pittel08} refined the results of Alon et al.\ in the special case where $\cG$ is a \emph{random} (uniformly chosen) $d$-regular graph on $n$ vertices for $d\geq 3$ fixed --- well-known to be a expander w.h.p.\ --- showing that the phase transition of $|\cC_1|$ mirrors that in the Erd\H{o}s--R\'enyi graph $\cG(n,p)$: w.h.p.,  $|\cC_1| \leq C(p) \log n$ at $p<p_c = \frac1{d-1}$ vs.\ $|\cC_1| \sim \theta_1 n$ at $p>p_c$ for an explicit $\theta_1(p)$, whereas 
 $|\cC_1|=n^{2/3+o(1)}$ at the critical $p_c$ (the precise order of $n^{2/3}$ and correct scaling of the critical window were subsequently found in~\cite{NP10}).

In this work we obtain the asymptotic size (lower bounded by $c(p,d)n$ above) and degree distribution of the giant component $\cC_1(G)$ for any high-girth expander $G\sim \cG_p$ at $p>1/(d-1)$, as well as its 2-core.

For $d\geq 3$ and $1<\lambda<d-1$, let 
\[ p=\lambda/(d-1)\]
 and let $q=q(\lambda,d)$ be the unique solution in $[0,1)$ of the equation
\begin{equation}\label{eq-q-def}
q = (1-p+pq)^{d-1}\,,
\end{equation}
well known (cf.~\cite{AthreyaNey04}) to coincide with the extinction probability of a $\Bin(d-1,p)$-Galton--Watson tree.

Recall that the 2-core of a connected component $\cC$, denoted here $\cC^{(2)}$, is its maximum induced subgraph with minimum degree at least $2$, and the (tree) excess of $\cC$ is the minimum number of edges that need to be removed from $\cC$ to turn it into a tree (note that the excess of $\cC$ equals that of $\cC^{(2)}$).
Our first main result establishes the asymptotic number of vertices and edges in the largest (giant) component in supercritical bond percolation on a high girth $d$-regular expander, as well as its 2-core, and consequently their asymptotic excess.
\begin{maintheorem}[the giant]\label{mainthm:giant-2core-sizes}
Fix $d\geq 3$ and $\frac{1}{d-1}<p<1$, and letting $q$ be as in~\eqref{eq-q-def}, define
\begin{align}\label{eq-theta-def} 
\theta_1 &:= 1 - q(1-p) - p q^2 \,, &\eta_1 := \tfrac12 p d (1-q^2)\,,\\
\label{eq-eta-def} 
\theta_2 &:= 1 -q - (d-1)p q (1-q) \,, &\eta_2 := \tfrac12 p d (1-q)^2 \,.
\end{align}
For every $\epsilon>0$ and $b>0$ there exist some $c,C,R>0$ such that, if $\cG$ is a regular $(b,d)$-expander on $n$ vertices with girth at least $R$, then 
with probability at least $1-Ce^{-c n}$, the random graph $G\sim \cG_p$ has 
\begin{align}
\label{eq-GC-V-E}
 \big| \tfrac1n |V(\GC)| - \theta_1 \big| < \epsilon \,,
&\qquad
 \big| \tfrac1n |E(\GC)| - \eta_1 \big| < \epsilon \,,\\
\label{eq-TC-V-E}
 \big| \tfrac1n |V(\TC)| - \theta_2 \big| < \epsilon \,,
&\qquad
 \big| \tfrac1n |E(\TC)| - \eta_2 \big| < \epsilon \,.
 \end{align}
 In particular, with probability at least $1-C\exp(-cn)$,  the excess of $\GC$ is within $2\epsilon n$ of $(\eta_1-\theta_1) n$, and the excess of $\TC$ is within $2\epsilon n$ of $(\eta_2-\theta_2) n$.
\end{maintheorem}

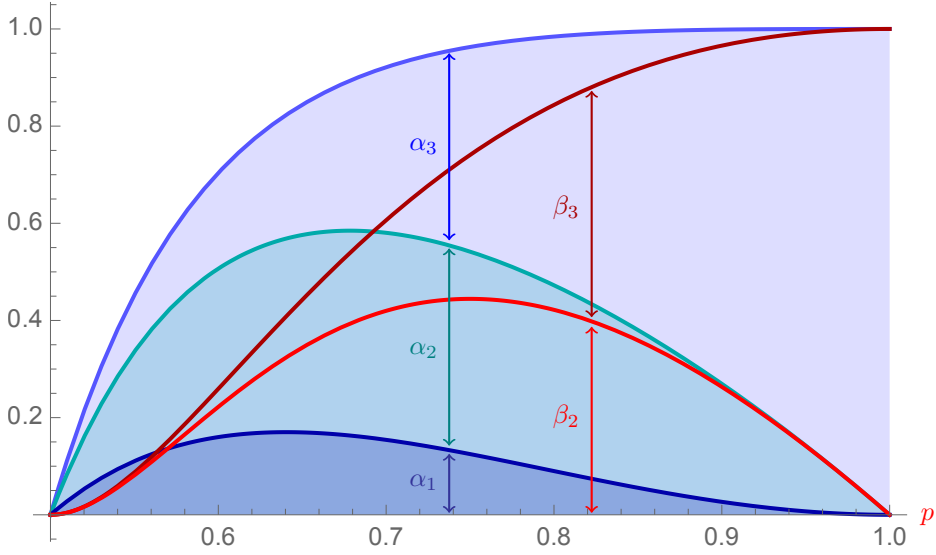
\begin{figure}
\begin{center}
\begin{tikzpicture}
    \node[scale=1] (plot) at (0,0) {};
    \node (fig) at (0,0) {
		\includegraphics[width=0.7\textwidth]{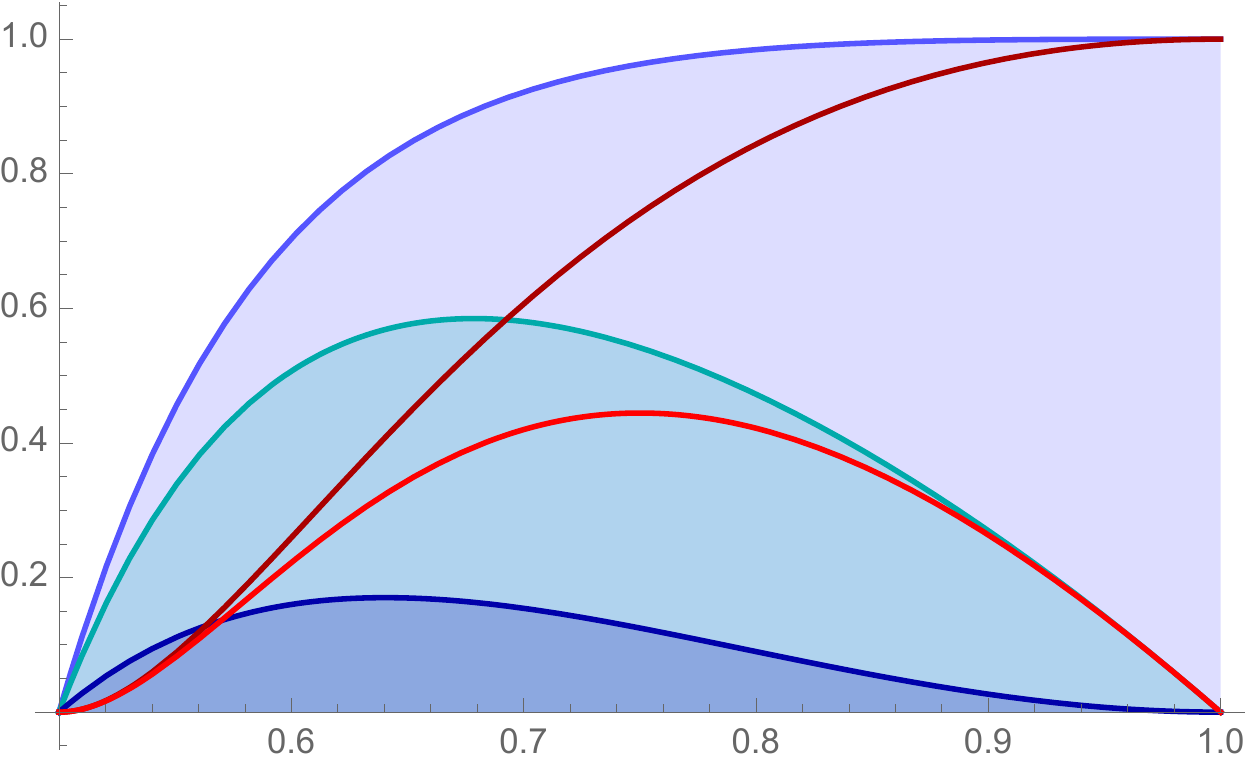}
	};
	\begin{scope}[shift={(plot.south west)},x={(plot.south east)},y={(plot.north west)}, font=\small]

     \draw[draw={rgb:red,85;green,85;blue,255}, thick, <->] (0,-11)--(0,-8.1);
     \node[font=\small,color={rgb:red,85;green,85;blue,255}] at (-1.25,-9.5) {$ \alpha_1$};
     \draw[draw={rgb:red,0;green,171;blue,171}, thick, <->] (0,-7.7)--(0,2);
     \node[font=\small,color={rgb:red,0;green,171;blue,171}] at (-1.25,-3) {$ \alpha_2$};
     \draw[draw={rgb:red,0;green,0;blue,171}, thick, <->] (0,2.4)--(0,11.6);
     \node[font=\small,color={rgb:red,0;green,0;blue,171}] at (-1.25,7) {$ \alpha_3$};

     \draw[draw=red, thick, <->] (7,-11)--(7,-1.85);
     \node[font=\small,color=red!66!black] at (5.75,4.) {$ \beta_3$};
     \draw[draw=red!66!black, thick, <->] (7,-1.38)--(7,9.75);
     \node[font=\small,color=red] at (5.75,-6.25) {$ \beta_2$};
     \node[font=\small,color=red] at (23.5,-11.2) {$ p$};
\end{scope}
\end{tikzpicture}
\end{center}
\vspace{-0.5cm}
\caption{Asymptotic degree distributions of the giant component ($\sim \!\alpha_k n$ degree-$k$ vertices (blue)) and its 2-core ($\sim\! \beta_k n$ degree-$k$ vertices (red)) in supercritical percolation on a high girth 3-regular expander.}
\label{fig:degree-dist}
\end{figure}

\begin{remark}[degree distributions of the giant and the 2-core]\label{rem:deg-seq}
In the setting of Theorem~\ref{mainthm:giant-2core-sizes}, we in fact asymptotically obtain the entire degree distributions of the giant component and 2-core (Theorem~\ref{thm:giant-2core-deg-dist}). For instance, for a fixed $b>0$, consider a sequence $\cG^{(n)}$ of regular $(b,3)$-expanders on $n$ vertices whose girth tends to $\infty$ (arbitrarily slowly) with $n$. For $\frac12<p<1$, let
\begin{align*} \alpha_1 = \frac3p (1-p)^2(2p-1)\,,\quad &\alpha_2 = \frac3{p^{2}}(1-p)(1-4p+6p^2-4p^3)\,,\quad &\alpha_3 = p^3\Big(1-\left(\tfrac{1-p}{p}\right)^6\Big)\,,\\
&\beta_2= \frac3{p^3}(1-2p)^2(1-p)\,,\quad &\beta_3 = \Big(\frac{2p-1}{p}\Big)^3	\,.
\end{align*}
Then w.h.p., the giant component $\GC$ of $G\sim\cG^{(n)}_p$ has $(\alpha_k+o(1))n$ vertices of degree $k$ for each $k\in\{1,2,3\}$, and its 2-core $\TC$ has $(\beta_k+o(1))n$ vertices of degree $k$ for each $k\in\{2,3\}$; see Figure~\ref{fig:degree-dist}.
\end{remark}

\begin{remark}[Limits for large $d$]\label{rem:complete-graph}
For $p=(1+\xi)/(d-1)$ with $\xi>0$ and $d\to\infty$ with $n$, one has 
\[ q\to 1-2\xi + \tfrac83 \xi^2 -\tfrac{28}9\xi^3 +  O(\xi^4) \] and
\begin{align*} 
\theta_1 \to 2\xi -\tfrac83\xi^2 + \tfrac{28}9\xi^3 + O(\xi^4) \,,\qquad
&\eta_1\to 2\xi -\tfrac83\xi^2 +\tfrac{34}9\xi^3 +O(\xi^4)\,,
&\eta_1-\theta_1 \to \tfrac23 \xi^3 + O(\xi^4)\,,
\\
\theta_2 \to 2\xi^2 -4\xi^3 +O(\xi^4)\,,\qquad
& \eta_2 \to 2\xi^2 -\tfrac{10}3\xi^3 + O(\xi^4)\,,
& \eta_2-\theta_2 \to \tfrac23\xi^3 + O(\xi^4)\,.
\end{align*}
Compare this to $\cG(n,p)$ for $p=\frac{1+\xi}n$ with $\xi=o(1)$: there $|V(\cC_1)|\sim 2\xi n$, $|E(\cC_1)|\sim 2\xi n$, $|V(\TC)|\sim 2\xi^2 n$ and the excess is $(1+o(1))\frac23\xi^3 n$ (see, e.g.,~\cite{Bollobas01,vdHofstad,JLR00} and the structure theorems in~\cite{DLP14,DKLP11}).
\end{remark}

In the context of the second largest component of $\cG_p$, it was shown by Alon et al.~\cite{ABS04} that, if $\cG$ is a $(b,d)$-expander (not necessarily regular) on $n$ vertices then there exists $\omega=\omega(b,d)<1$ such that, w.h.p., for every sequence $0<p_n<1$,
\[ \lim_{n\to\infty}\P(|\cC_2(\cG_{p_n})|<n^\omega)=1\,.\]
Indeed, it is well-known (cf., e.g.,~\cite{Bollobas01}) that in the Erd\H{o}s--R\'enyi random graph $\cG(n,p)$ (where $\cG$ is the complete graph) such is the case, as the supercritical regime $p>1/n$ admits a single giant component w.h.p., and all other components are logarithmic in size (see, e.g.,~\cite{Bollobas01} and~\cite{JLR00}). Such is also the case in percolation on random regular graphs (cf.~\cite{NP10}). In light of this --- and in line with results of Pittel~\cite{Pittel08} on $|\cC_2|$ in random regular graphs --- one may believe that the same holds for percolation on expanders, whereby the $n^{\omega}$ above could be replaced by some $C(b,d) \log n$ with probability arbitrarily close to $1$.

Perhaps surprisingly, it turns out that even on a family of regular expanders with arbitrarily large girth, the above polynomial bound of $n^\omega$ is essentially best possible, as the second largest component can have size $n^{\omega'}$ with $\omega'$ arbitrarily close to $1$.
\begin{maintheorem}[second component]\label{mainthm:second-comp}
For every $d\geq 3$, $R\geq 1$, $p\in(\frac1{d-1},1)$ and $\alpha\in(0,1)$ there exist $b>0$ and a regular $(b,d)$-expander $\cG$ on $n$ vertices with girth at least $R$ where $G\sim\cG_p$ has $|V(\mathcal{C}_2)|\gtrsim n^\alpha$~w.h.p.
\end{maintheorem}

\begin{remark}
Using essentially the same construction, for any sequence $0<\alpha_1\leq \alpha_2\leq \ldots\leq \alpha_k< 1$ with $k$ fixed, one can construct an expander $\cG$ such that w.h.p.\ $G\sim \cG_p$ has components whose sizes have orders $n^{\alpha_1},\ldots,n^{\alpha_k}$, respectively (in addition to the linearly sized component).
\end{remark}

Towards a construction of an infinite graph that disproves the existence of ``expanders at all scales,'' Benjamini asked the following question (a positive answer to which would be a step in said construction), on the relation between the diameter of the connected component of $v$ in $\cG_p$ to the existence of a giant.

 \begin{question*}[Benjamini~{\cite[Q.~5.5]{Benjamini13}}] Let $\cG$ be a bounded degree expander. Further assume that there is a fixed vertex $v\in \cG$, so that $G\sim\cG_{1/2}$ satisfies
\[ \P\big(\diam(\cC_v(G)) > \tfrac12\diam(\cG)\big) > \tfrac12\,.\]
Is there a giant component w.h.p.?
\end{question*}
A variant of the construction in Theorem~\ref{mainthm:second-comp} gives a negative answer to this question.
\begin{maintheorem}\label{mainthm:q-benjamini}
For every $\epsilon>0$ and $0<p<1$ there exist $b,d,\delta>0$ and, for infinitely many values of~$n$, a $(b,d)$-expander $\cG$ on $n$ vertices with a prescribed vertex $v$, such that the graph $G\sim\cG_p$ satisfies
\[ \P\big(\diam(\cC_v(G)) \geq (1-\epsilon)\diam(\cG)\big) \geq 1-\epsilon\,,\]
and yet there are no components of size larger than $n^{1-\delta}$ in $G$ w.h.p.
\end{maintheorem}

Finally, in~\S\ref{sec:separators_long_paths}, we explore additional typical properties of percolation on high-girth expanders (where we find the typical order of those random variables but not their precise asymptotics): the existence of linear simple paths (Corollary~\ref{cor-long_path}), the fact that linearly many vertices must be removed in order to disconnect the giant component (Theorem~\ref{thm-separators}), and the existence of complete minors (Corollary~\ref{cor-complete_minor}).

\subsubsection*{Main techniques.}\vspace{-5pt}

The starting point in proving Theorem~\ref{mainthm:giant-2core-sizes}, and furthermore its more detailed version Theorem~\ref{thm:giant-2core-deg-dist} which gives the typical degree profile in the giant and 2-core of $\cG_p$, is a refinement of the elegant sprinkling argument of Alon et al.~\cite{ABS04}.  While that argument was used in~\cite{ABS04} and later also in~\cite{BNP11} to show that w.h.p.\ a \emph{fixed proportion} of the vertices whose local neighborhoods in~$\cG_{p-\epsilon}$ are ``large'' will  belong to the giant after adding $\cG_{\epsilon}$, in fact \emph{all but a negligible proportion} of such vertices will be in the giant. Namely, the local property of whether the connected component of a vertex $x$ has size at least $R$ turns out to be a predictor, for all but $\epsilon(R) n$ vertices, as to whether  $x \in \cC_1(\cG_p)$  (similarly for edges) --- see Proposition~\ref{prop:E1-V1-approx}. 
One can then readily use this local predictor to read off the degree profile, and more generally, the entire local geometry in $\GC(\cG_p)$ for a high girth expander $\cG$, as demonstrated next.

\vspace{-1pt}
\begin{example}\label{ex:subgraphs}[Density of local neighborhoods in $\cC_1$]\label{exmample:arbitrary-subgraphs}
Let $G\sim \cG_p$ for $\frac1{d-1}<p<1$ and a $d$-regular expander $\cG$ on $n$ vertices whose girth tends to $\infty$ with $n$.
Set $q$ as in~\eqref{eq-q-def}, let $T$ be a rooted tree with $k$ levels and $\ell$ leaves (for some fixed $k,\ell$), and let $\alpha_T = \P(\mathcal{T}_p \simeq T)$, where $\mathcal{T}$ is the tree on $k$ levels that is $d$-regular except at its leaves (and $\simeq$ denotes graph isomorphism).
The number of vertices in $\cC_1(G)$ whose $k$-radius neighborhood is isomorphic to $T$ is then w.h.p.\ asymptotically $\left(1-q^\ell\right)\alpha_T \, n$.
\end{example}
\vspace{-7pt}
\begin{example}\label{ex:paths}[Asymptotic density of paths in $\cC_1$]
For $G$ as in Example~\ref{ex:subgraphs} and for any fixed $\ell\geq1$, the number of paths with $\ell$ edges in $\cC_1$ is w.h.p.\ asymptotically
$\frac12d(d-1)^{\ell-1} p^{\ell} \big(1-q^{\ell+1}(1-p+pq)^{1-\ell}\big)n$.
\end{example}
\vspace{-1pt}

The analysis of cycles in $G$ (and its 2-core) is substantially more delicate. Naively, one might expect our results to suggest that, if a certain edge $xy$ in $G\sim\cG_p$ is such that in the graph $H_{xy}:= G\setminus \{xy\}$ the components of $x$ and $y$ are both ``large,'' then both should typically belong to the giant in $H$ (which is a subset of the giant of $G$), hence the edge $xy$ should lie on a cycle in $\GC(G)$. However, turning this intuition into a rigorous argument is problematic in light of the fact that we cannot actually delete $xy$ (as we aim to carry this analysis simultaneously for all edges $xy$), and our mechanism of securing that the components of $x$ and $y$ would belong to the giant was non-constructive (arguing that there are small cuts whence sprinkling would patch most such components together); in particular, that argument gave no control over whether or not the edge $xy$ itself is a bridge as opposed to a cycle edge in $\TC(G)$.

To remedy this, we introduce the notion of a \emph{$k$-thick} set --- roughly put, a set that can be covered by disjoint connected components of size at least $k$ each (see Definition~\ref{def:k-thick}) --- and show, in what may be of independent interest, that if $G$ is an expander and $H$ is a slightly percolated subgraph of $G$ (that is, $H\sim G_{1-\delta}$ for some small $\delta>0$), then every linearly-sized $k$-thick set expands in $H$ (cf.\ Claim~\ref{clm:G-red}). Since we seek to analyze the effect of sprinkling on components that are ``large,'' such expansion suffices, and does in fact hold in $H$, supporting the above framework of the proof. 

Finally, our constructions in Theorem~\ref{mainthm:second-comp} and~\ref{mainthm:q-benjamini} exploit the source of some of the obstacles described above: while we have full understanding of the geometry of $G$ in the microscopic scale (locally the graph is a regular tree up to an arbitrarily large radius), and some control over it in the macroscopic scale (the expansion implies bounds on the number of edges between every two linearly-sized sets), we have limited control over it in the intermediate scales. For instance, if the $k$-radius ball around a vertex $x$ grows as $(d-1)^k$ then percolation is supercritical, as opposed to growing only as $b^k$ whence it might be subcritical (yet still satisfying the expansion property). Our constructions combine both these types of expansion in the mesoscopic scales to create components of size $\Theta(n^{\alpha})$ ($0<\alpha<1$) in~$\cG_p$.
 
\section{Asymptotics of the giant component and its 2-core}\label{sec:asymp}

The goal in this section is to prove the following result (from which Theorem~\ref{mainthm:giant-2core-sizes} follows) describing the typical degree distribution of the giant component and its 2-core in $\cG_p$ up to an error of $o(n)$.

\begin{theorem}\label{thm:giant-2core-deg-dist}
Fix $d\geq 3$ and $1<\lambda<d-1$, let $p=\lambda/(d-1)$, and with $q$ as in~\eqref{eq-q-def}, define
\begin{align}\label{def:alpha_k}
 \alpha_k &= \binom{d}k p^k (1-p)^{d-k} (1-q^k) & (k=1,\ldots,d)\,,\\	
 \label{def:beta_k}
 \beta_k &= \binom{d}{k}   \frac{p^k (1-q)^k q}{\left(1-p+pq\right)^{k-1}}  &(k=2,\ldots,d)\,.
  %
\end{align}
For a given $d$-regular graph $\cG$, let $D_k$ ($1 \leq k\leq d$) be the number of degree-$k$ vertices in $\GC$, the largest component of the graph $G\sim\cG_p$, and let $D^*_k$ denote the number of degree-$k$ vertices in its 2-core $\TC$. 
Then for every $\epsilon>0$ and $b>0$ there exist some $c,R>0$ so that, if $\cG$ is a regular $(b,d)$-expander on $n$ vertices with girth at least $R$, then with probability at least $1-\exp(-c n)$,
\[ \big| D_k/n - \alpha_k\big|<\epsilon \quad\mbox{ for all $1\leq k \leq d$}\qquad\mbox{ and }\qquad
\big| D^*_k/n - \beta_k\big|<\epsilon\quad \mbox{ for all $2\leq k \leq d$}\,.\]
\end{theorem}

Indeed, Theorem~\ref{mainthm:giant-2core-sizes} follows from verifying (recalling~\eqref{eq-q-def}) 
that the quantities $\theta_i$ and $\eta_i$ ($i=1,2$), as defined in~\eqref{eq-theta-def}--\eqref{eq-eta-def}, satisfy the following w.r.t.\ the above defined $\alpha_k$ ($1\leq k\leq d$) and $\beta_k$ ($2\leq k\leq d$):
\[ \mbox{$\theta_1 = \sum_{k=1}^d \alpha_k\,,\quad \eta_1 = \frac12\sum_{k=1}^d k \alpha_k\,,\quad \theta_2 = \sum_{k=2}^d \beta_k\,,\quad  \eta_2 = \frac12 \sum_{k=2}^d k\beta_k$}\,.\]
 
\subsection{The giant component}\label{sec:giant}

Fix $\epsilon>0$ small enough so that $p':=p-\epsilon$ satisfies $p' > 1/(d-1)$. There exists $R_0$ so that, for all $R\geq R_0$, the probability of survival to depth $R$ in a Galton--Watson (GW) tree with offspring distribution $\Bin(d-1,p')$  is at least $1-q'$, and the corresponding probability in a GW-tree with offspring distribution $\Bin(d-1,p)$ is at most $1-q+\epsilon$.
Since $\lambda\mapsto q(\lambda)$ is continuous, one has $q'\uparrow q$ as~$\epsilon\downarrow 0$.   
We will couple $G'\sim \cG_{p'}$ and $G\sim \cG_p$ by letting $E(G) = E(G')\cup E(F')$ for $F'\sim \cG_{\epsilon'}$, where 
\[ \epsilon' := \frac{p-p'}{1-p'}  = \frac{\epsilon}{1-p'}\,.\]

For a graph $H$ and an ordered pair of vertices $x,y\in V(H)$, define
\begin{equation}\label{eq:def-Axy} \mathcal{A}^R_{x,y}= \mathcal{A}^R_{x,y}(H) = \left\{ |\cC_y(H\setminus \{xy\}) | \geq R \right\} \end{equation}
(where $H\setminus\{xy\}$ is obtained by deleting the edge $xy$ if present), and for an unordered pair  $x,y\in V(H)$~let
\begin{equation}\label{eq:def-Bxy} \mathcal{B}^R_{xy} =\mathcal{B}^R_{xy}(H)= \left(\mathcal{A}^R_{x,y} \cup  \mathcal{A}^R_{y,x}\right) \cap \{xy\in E(H)\}\,; \end{equation}
i.e., $\cA_{x,y}^R$ says that after removing $xy$ (if present) the component of $y$ has at least $R$ vertices, while $\cB_{xy}^R$ says that, in addition to this, after removing $xy$ the component of $x$ also has at least $R$ vertices, and that the edge $xy$ does belong to $H$. 
Further define, for every vertex $x\in V(H)$,
\begin{equation}\label{eq:def-Bx} \cB^R_x = \bigcup_{y} \left(
\mathcal{A}^R_{x,y} \cap \{xy\in E(H)\}\right)\,.\end{equation}
Finally, let
\begin{equation}\label{eq-U1-def} 
\begin{array}{l}
E_1(H) = \{ xy\in E(H) \,:\; \cB^R_{xy}(H)\mbox{ holds}\}\,,\\ 
\noalign{\medskip}
V_1(H) = \{ x\in V(H)\,:\; \cB^R_x(H)\mbox{ holds}\}\,.
\end{array}
\end{equation}
The main result we wish to prove in this subsection is as follows.
\begin{proposition}\label{prop:E1-V1-approx}
For every $\epsilon,b>0$ there exist $R$ and $c>0$ such that, if $\cG$ is a regular $(b,d)$-expander on $n$ vertices with girth greater than $2R$, and $G\sim \cG_p$, then
\begin{align}
\P\Bigl( \bigl| E_1(G) \xor E(\GC(G)) \bigr| > \epsilon n\Bigr) &\leq \exp(-c n)\,,\label{eq:approx-E1}\\
\P\Bigl( \bigl| V_1(G) \xor V(\GC(G)) \bigr| > \epsilon n\Bigr) &\leq \exp(-c n)\,.\label{eq:approx-V1}
\end{align}
\end{proposition}
\begin{proof}
Observe that, for any graph $H$ with maximum degree $d$, if $|\cC_x| > d R$ then the vertex $x$ must be incident to some $y$ such that $\cA_{x,y}^R$ holds; similarly, if an edge $xy$ belongs to a component of size at least $2R$ then at least one of the events $\cA_{x,y}^R,\cA_{y,x}^R$ must hold. 
 That is,
\begin{equation}
\label{eq:approx-C1}
\begin{array}{r} \bigcup\left\{ E(\cC)  \,:\;\mbox{$\cC$ is a connected component of $H$ with $|\cC|\geq 2R$}\right\}  \subseteq E_1(H)\,.\\
 \bigcup\left\{ V(\cC)  \,:\;\mbox{$\cC$ is a connected component of $H$ with $|\cC|> d R$}\right\}  \subseteq V_1(H)\,.
 \end{array}	
\end{equation}

By the assumptions that $R\geq R_0$ for a large enough $R_0$ and that the girth is greater than $2R$,
\[1-q \leq \P(\mathcal{A}^R_{x,y}(G)) \leq 1-q+\epsilon\,,\qquad 
1-q'\leq \P(\mathcal{A}^R_{x,y}(G')) \leq 1-q'+\epsilon\,.\]
When $H\sim\cG_p$ or $H\sim \cG_{p'}$, the standard edge-exposure martingale (see, e.g.,~\cite[Sec.~2.4]{JLR00}) --- noting that adding/deleting an edge influences at most $2(d-1)^{R}$ edges --- implies via Hoeffding--Azuma that
\begin{equation}\label{eq:E1-concentrate} \P\left(\Bigl||E_1(H)|-\E [|E_1(H)|]\Bigr| \geq a\right) \leq \exp\left(-\frac{a^2}{4 d n (d-1)^{2R}}\right)\,.	
\end{equation}
and similarly,
\begin{equation}\label{eq:V1-concentrate} \P\left(\Bigl||V_1(H)|-\E [|V_1(H)|]\Bigr| \geq a\right) \leq \exp\left(-\frac{a^2}{4 d n (d-1)^{2R}}\right)\,.	
\end{equation}
Note that, in $G\sim\cG_p$, since the girth of $\cG$ is greater than $2R$ (hence  $\mathcal{A}^R_{x,y}$ and $\mathcal{A}^R_{y,x}$ are independent),
\[ 
p(1-q^2) \leq \P(\mathcal{B}^R_{xy}) \leq  p (1-(q-\epsilon)^2)\,,\]
and, using that $q=(1-p+pq)^{d-1}$ and the mutual independence of $\cA_{x,y}^R$ for all neighbors $y$ of $x$ in $\cG$,
\[ 1-q(1-p+p q) \leq \P(\mathcal{B}^R_x) \leq 1- (1-p + p(q-\epsilon))^d\,.\]
We can therefore deduce that 
\begin{align}\label{eq:E1-mean}
 \tfrac12 pd (1-q^2) n &\leq  \E [|E_1(G)|] \leq \tfrac12 pd (1-(q-\epsilon)^2) n\,,\\
 \label{eq:V1-mean}
  \bigl(1-q(1-p+p q)\bigr) n &\leq  \E [|V_1(G)|] \leq \bigl(1- (1-p + p(q-\epsilon))^d\bigr)n\,.
\end{align}
Combining this with~\eqref{eq:approx-C1}--\eqref{eq:E1-concentrate} implies that, in $G\sim\cG_p$, with probability at least $1-\exp(-cn)$,
\begin{align}
\label{eq:upper-bound-E1}
 \sum_{\cC : |V(\cC)|\geq2R } |E(\cC)|  &\leq |E_1(G)| \leq \left(\tfrac12 pd (1-(q-\epsilon)^2) + \epsilon\right) n \,,\\
 \label{eq:upper-bound-V1}
 \sum_{\cC : |V(\cC)|>dR } |V(\cC)| & \leq |V_1(G)| \leq \bigl(1- (1-p + p(q-\epsilon))^d + \epsilon \bigr)n\,.
\end{align}
(In particular, this gives upper bounds on $|E(\cC_1(G))|$ and on $|V(\cC_1(G))|$.)

Next, we consider $G'$, and note that using~\eqref{eq:E1-concentrate}--\eqref{eq:V1-concentrate}, together with the analogs of~\eqref{eq:E1-mean}--\eqref{eq:V1-mean} for $G'$, yields
\begin{align}
	\label{eq:E1(Gp')}
	\P\left(|E_1(G')| < (\tfrac12 p'd (1-q'^2) - \epsilon)n \right) &\leq \exp(-c n)\,,\\
	\label{eq:V1(Gp')}
		\P\left(|V_1(G')| < (1-q'(1-p'+p'q') - \epsilon)n \right) &\leq \exp(-c n)\,.
\end{align}

\begin{claim}\label{clm:menger}
For every $\epsilon,b,d>0$ there exist $c,R>0$ such that the following holds for large enough~$n$.
If $\cG$ is a regular $(b,d)$-expander with $n$ vertices, and $\mathcal{S}$ is a collection of disjoint vertex subsets of $\cG$, each of size  at least $R$, then the probability that there exist two subsets $\mathcal{S}_1,\mathcal{S}_2\subset \mathcal{S}$, with a total of at least $\epsilon n$ vertices in each, and no path between them in $H\sim \cG_{\epsilon}$, is at most $\exp(-c n)$.
\end{claim}
\begin{proof}
	By Menger's Theorem and our hypothesis on the edge expansion of $\cG$, for every two disjoint subsets $A,B\subset V(\cG)$ of size at least $\epsilon n$ each, there are at least $b\epsilon n$ edge-disjoint paths between them in $\cG$. 
	Since the total number of edges in $\cG$ is $dn/2$, it follows that for every two such subsets $A,B$,
	\begin{equation}\label{eq:short-paths-in-G}
\mbox{ there are at least } \left\lceil\tfrac{b\epsilon}2 n\right\rceil \mbox{ edge-disjoint paths of length at most } \left\lfloor \tfrac{d}{b\epsilon}\right\rfloor\mbox{ between }A,B\mbox{ in }\cG\,.
	\end{equation}
In particular, this holds for every two subsets $\mathcal{S}_1$ and $\mathcal{S}_2$ of $\mathcal{S}$ each containing at least $\epsilon n$ vertices. The probability that none of these short paths between $\mathcal{S}_1$ and $\mathcal{S}_2$ survive in $\cG_\epsilon$ is at most
	\[ \left(1-\epsilon^{d/(b \epsilon)}\right)^{\frac12 b\epsilon n} \leq \exp\left[ -\tfrac12 b \epsilon^{1+d/b\epsilon} n\right]\,.\]
	Altogether, a union bound over at most $2^{2n/R}$ possible pairs of subsets of $\mathcal{S}$ shows that the probability that there exist $\mathcal{S}_1,\mathcal{S}_2$ violating the statement of the claim is at most 
	\[ \exp\left[ \left(R^{-1}2\log 2  -\tfrac12 b \epsilon^{1+d/b\epsilon} \right)n\right] \,.\]
	Taking $R$ large enough completes the proof.
\end{proof}

\begin{corollary}\label{cor:V1-after-sprinkling}
For every $\epsilon,b,d>0$ there exist $c,R>0$  so that the following holds for large enough~$n$. If $\cG$ is a regular $(b,d)$-expander on $n$ vertices with girth greater than $2R$, then  with probability $1-O(e^{-c n})$ there exists a connected component $\cC$ of $G'\cup F'$ containing all but at most  $2\epsilon n$ vertices of the set \[\overline{V}_1(G') := \left\{ y \in V(G) \,:\; \cA_{x,y}^{R}(G')\mbox{ holds for some }x\in V(G)\right\}\,.\] In particular, $\cC$ contains all but at most $2\epsilon n$ of the vertices $V_1(G') $.
\end{corollary}
\begin{proof}
Let $\mathcal{S}(G') = \{ \cC_y(G') \,:\; y \in \overline{V}_1(G') \}$, noting that every component in $\mathcal{S}(G')$ is of size at least $R$ by definition.
By Claim~\ref{clm:menger}, with probability at least $1-\exp(-cn)$, one cannot partition $\mathcal{S}(G')$ into two subsets $\mathcal{S}_1,\mathcal{S}_2$, each of size at least $\epsilon n$, such that there will be no path connecting them in $G=G'\cup F'$. With this in mind, let $U$ be a set of vertices obtained by collecting (in an arbitrary way) connected components of $\mathcal{S}(G')$ in $G$, until $|U \cap \overline{V}_1(G')|\geq \epsilon n$. Note that $|U\cap \overline{V}_1(G')| < \epsilon n + |\cC\cap \overline{V}_1(G')|$, where $\cC$ is a connected component in $G$ (the last one that joined $U$). 
 If $|\cC\cap \overline{V}_1(G')| \leq |\overline{V}_1(G')| - 2\epsilon n$, then the  sets $\mathcal{S}_1= U\cap \overline{V}_1(G')$ and $\mathcal{S}_2=\overline{V}_1(G')\setminus U$ (each of size at least $\epsilon n$) violate the above property. 
\end{proof}

Note that, on the event stated in the above corollary,  the component $\cC$  contains all but at most $2\epsilon d n$ edges of $E_1(G')$ (losing at most $d$ edges per vertex in $V_1(G')\setminus \cC$). Therefore, using~\eqref{eq:E1(Gp')}, 
with probability at least $1-\exp(-c n)$ we have that $\cC$ is of linear size, and moreover,
\[ |E(\cC)| > \tfrac12 p'd(1-q'^2)n  - (2d+1)\epsilon  n\,.\] 
In particular,  the component $\cC$ belongs to the set of components in the left-hand side of~\eqref{eq:approx-C1}, so 
\[ E(\cC) \subseteq E_1(G)\,,\qquad V(\cC)\subseteq V_1(G)\,.\]
By~\eqref{eq:upper-bound-E1}, the total number of edges in all components $\cC'\neq \cC$ with $|\cC'|\geq 2R$ is at most
\[  |E_1(G) \setminus E(\cC)|\leq \left(\tfrac12pd(1-(q-\epsilon)^2) - \tfrac12 p'd(1-q'^2) \right)n + (2d+2)\epsilon n\,.\]
For small enough $\epsilon>0$, the right-hand is at most $\tilde\epsilon n$, where $\tilde\epsilon \downarrow 0 $ as $\epsilon\downarrow 0$. The proof of~\eqref{eq:approx-E1} is therefore concluded by the fact that, for a small enough $\epsilon>0$, one has $\tilde\epsilon < \eta_1/2$, whence the total number of edges in components $\cC'\neq \cC$ with $|\cC'|\geq2R$ is strictly less than that in $\cC$, thus in fact $\cC=\cC_1$.

Finally, to establish~\eqref{eq:approx-V1}, recall that, by~\eqref{eq:V1(Gp')}, with probability at least $1-\exp(-cn)$ we have
\[ |V(\cC)| \geq |V_1(G')| - 2\epsilon n \geq (1-q'(1-p'+p'q')-3\epsilon)n\,.\] 
Comparing this with the upper bound on $|V_1(G)|$ in~\eqref{eq:upper-bound-V1}, and recalling that $p'\uparrow p$ and $q'\uparrow q$ as $\epsilon\downarrow 0$, thereby concludes the proof of Proposition~\ref{prop:E1-V1-approx}.
\end{proof}

From the proposition and~\eqref{eq:E1-concentrate},\eqref{eq:E1-mean}, we deduce the required estimate on $|E(\GC)|$, while~\eqref{eq:V1-concentrate},\eqref{eq:V1-mean} analogously imply the required estimate on $|V(\GC)|$. 

Furthermore, from~\eqref{eq:approx-E1} we see that with probability $1-\exp(-cn)$, at most $2\epsilon n$ vertices in $G\sim\cG_p$ have a discrepancy between their degree in $\GC$ and that in $E_1(G)$. The statement of Theorem~\ref{thm:giant-2core-deg-dist} that $|D_k/n-\alpha_k|<\epsilon$ thus follows from the fact that, for every $x$ and $1\leq k \leq d$, the probability that $\#\{y : xy\in E_1(G)\}=k$ corresponds to $\alpha_k$ (up to replacing $1-q$ by $1-q+\epsilon$ in that expression),
 as this event occurs iff $x$ has exactly $k$ neighbors in $G$, out of which  there exists at least one vertex $y$ which satisfies $\{ |\cC_{y}(\GC\setminus\{x y\})|\geq R-1\}$ (up to replacing $R-1$ by $R$ in the case $k=1$).

\subsection{The 2-core of the giant component}\label{sec:2-core}
Recall the definition of $\cA_{x,y}$ in~\eqref{eq:def-Axy}, and define
\begin{equation}\label{eq-U2-def} 
\begin{array}{l}
E_2(H) = \{ xy\in E(H) \,:\; \cA_{x,y}^R(H)\cap \cA^R_{y,x}(H)\cap \{xy\in E(H)\} \}\,,\\ 
\noalign{\medskip}
V_2(H) = \{ x\in V(H)\,:\; xy\in E_2(H)\mbox{ for some $y\in V(H)$}\}\,.
\end{array}
\end{equation}
The main result we wish to prove in this subsection is the following characterization of the 2-core of a typical random graph $G \sim \cG_p$. It consists of a local rule for inclusion of vertices and edges in $\TC(G)$, the 2-core of its largest component, which determines it up to at most $\epsilon n$ vertices. We also show that all other components contribute a combined total of at most $\epsilon n$ vertices to the 2-core of $G$. (It is easy to see that, in both cases, a linear error of some $\epsilon' n$ vertices must be allowed, e.g., when every vertex is part of a cycle of length $O(R)$.) 
\begin{proposition}\label{prop:E2-V2-approx}
For every $\epsilon,b>0$ and $d\geq 3$ there exist some $R$ and $c>0$ such that, if $\cG$ is a regular $(b,d)$-expander on $n$ vertices with girth greater than $2R$, and $G\sim \cG_p$, then
\begin{align}
\P\Bigl( \bigl| E_2(G) \xor E(\TC(G)) \bigr| > \epsilon n\Bigr) &\leq \exp(-c n)\,,\label{eq:approx-E2}\\
\P\Bigl( \bigl| V_2(G) \xor V(\TC(G)) \bigr| > \epsilon n\Bigr) &\leq \exp(-c n)\,,\label{eq:approx-V2}
\end{align}
whereas
\begin{align}
\P\Bigl( \sum_{i\geq 2}\bigl| \cC_i^{(2)}(G)  \bigr| > \epsilon n\Bigr) &\leq \exp(-c n)\,.
\label{eq:2core-out-of-gc}	
\end{align}
\end{proposition}
\begin{proof}
First observe that, for any graph $H$ with girth greater than $2R$,
\[ \bigcup \left\{ E\big(\cC^{(2)}\big)\,:\; \cC=\cC_x(H) \mbox { for some $x\in V(H)$}\right\} \subset E_2(H)\,;\] 
indeed, for any such  $H$ and edge $xy\in E(H)$, if $\cA^R_{x,y}$ does not hold then the component $\cC_{y}(H\setminus\{xy\})$ has size less than $R$, and hence it is a tree by the girth assumption, so $xy$ cannot  belong to the 2-core. In particular, 
\begin{equation}\label{eq-C2-in-E2}
\bigcup_{i\geq 1} E\big(\cC^{(2)}_i(G)\big) \subset E_2(G)\quad \mbox{ and } \quad
\bigcup_{i\geq 1} V\big(\cC_i^{(2)}(G)\big) \subset V_2(G)
\,.	
\end{equation}
Establishing the following bound (for some $R,c$ depending on $b,d,\epsilon$) will allow us to conclude the proof:
\begin{equation}\label{eq-E2-in-C2+}
	 \P\left( | E(\TC(G)) | > (1- \epsilon) \tfrac12 dp (1-q)^2 n\right) = 1-O(\exp(-cn))\,.
	  \end{equation}
In order to see that this indeed implies the statement of the proposition, note that, as in the argument preceding~\eqref{eq:E1-concentrate}, applying Hoeffding's inequality to the appropriate edge-exposure martingale (where the Lipschitz constant is as before) implies that, if $G\sim \cG_p$ then
\begin{equation}\label{eq:E2-concentrate} \P\left(\left|E_2(G)|-\E[|E_2(G)|]\right|\geq a\right) \leq \exp\left(-\frac{a^2}{4 d n (d-1)^{2R}}\right)\,,
\end{equation}
where
\begin{equation}\label{eq:E2-size} \tfrac12 d p (1-q)^2 n \leq   \E[|E_2(G)|]\leq \tfrac12 d p (1-q+\epsilon)^2 n \,;\end{equation}
thus, the combination of \eqref{eq-C2-in-E2}--\eqref{eq-E2-in-C2+} will indeed imply~\eqref{eq:approx-E2}, and in particular will also give~\eqref{eq:approx-V2}, as well as the upper bound~\eqref{eq:2core-out-of-gc} on the cumulative size of the 2-cores of  $\cC_i^{(2)}(G)$ for $i>1$. 

It remains to prove~\eqref{eq-E2-in-C2+}. 
As before, let $p'=p-\epsilon > 1/(d-1)$ where $\epsilon'=\epsilon/(1-p')$ such that, if we take
$G'\sim \cG_{p'}$ and  $F'\sim \cG_{\epsilon'}$ then we can produce $G\sim \cG_p$ via $E(F')\cup E(G')$. The key to the proof would be to randomly partition $E(\cG)$ into $E_\textsc{b}\cup E_\textsc{R}$ --- denoting $E_\textsc{b}$ as \textsc{blue} edges and  $E_\textsc{r}$ as \textsc{red} edges --- and then connect the appropriate proportion of edges $(x,y)\in E(\cG)$ in the giant component $\GC(G)$ via one \textsc{blue} edge $e$ and one \textsc{red} path to establish that $e\in E(\TC(G))$. In what follows, if $H$ is a subgraph of $\cG$, we let $H_\textsc{b}$ denote the (\textsc{blue}) subgraph of $H$ whose edges are $E(H)\cap E_\textsc{b}$ and, similarly, we let  $H_\textsc{r} $ denote the (\textsc{red}) subgraph of $H$ whose edges are $E(H)\cap E_\textsc{r}$.

\begin{remark}
In principle, to show that $e\in E_2(G')$ is also in $E(\TC(G))$, one would only need to show that $e\in \GC$ and provide a path that connects the two clusters --- each of size at least $R$ --- that are at its endpoints \emph{without using said edge}. However, we must resort to witnesses in the form of a \textsc{red} path between the endpoints $x,y$ of the edge $e$ which itself is \textsc{blue} due to our mechanism to guarantee that $e\in \GC$: the latter uses sprinkling (new edges in $F'$) to connect large clusters (ones of size at least $R$) to one another, and potentially might use, for instance, the edge $e$ to connect the clusters of $x$ and $y$. To remedy this, we independently color each edge of $\cG$ in \textsc{blue} with some probability $\hat\epsilon$, for a lower bound on the number of \textsc{blue} edges in $\TC(G)$, which will turn out to be asymptotically tight, despite insisting on sprinkling only \textsc{red} edges to connect the large clusters (see Remark~\ref{rem:sprinkling-red}).
\end{remark}

The random partition $E(\cG)=E_\textsc{b}\cup E_\textsc{r}$ is defined as follows: we let each $e\in E(\cG)$, independently, belong to $E_\textsc{b}$ with probability $\hat\epsilon$, given by
\begin{equation}
\label{eq:epsilon-hat-def}	
\hat\epsilon := 3d\sqrt{\epsilon}\,,
\end{equation}
and further suppose that
\begin{equation}
	\label{eq:epsilon-hat-small}
 \hat\epsilon < \tfrac14 (24d)^{-2/b}\,,
\end{equation}
which we may assume w.l.o.g.\ since the event we wish to estimate in~\eqref{eq-E2-in-C2+} is monotone in $\epsilon$.

Our goal will be to show that the edges in the set
\[ \widehat{\mathcal{U}}:= \left\{ e\in E(G'_\textsc{b}) \,:\; \cA^R_{x,y}(G'_\textsc{r}) \cap \cA^R_{y,x}(G'_\textsc{r})\mbox{ holds}\right\}
\]
represent, up to an arbitrarily small error, an $\hat\epsilon$-proportion of the 2-core $\TC(G)$. 
This would entail adapting our strategy of connecting small components via sprinkling to be restricted to \textsc{red} edges (so we could guarantee $e\in \widehat{\mathcal{U}}$ would be part of a cycle), towards which we introduce the following notion.
\begin{definition}[$k$-thick subsets]\label{def:k-thick} We say a subset $S$ of vertices of a graph $H$ is $k$-thick if there exists a collection $\{S_i\}$ of disjoint connected subsets of $H$, each of size at least $k$, such that $S = \bigcup S_i$.
\end{definition}
The idea behind this definition is that, although $\cG_\textsc{r}$ is not an expander --- for instance, it contains a linear proportion of isolated vertices --- w.h.p., sets that are $k$-thick maintain edge expansion in $\cG_\textsc{r}$:
\begin{claim}\label{clm:G-red}
Let $\cG$ be a regular $(b,d)$-expander on $n$ vertices,  let  $\epsilon>0$ and set $E_\textsc{b},E_\textsc{r}$ as above for  $\hat\epsilon$ as in~\eqref{eq:epsilon-hat-def}. There exists $k(\epsilon,b,d)$ such that, with probability $1-O(2^{-\epsilon n})$,
\begin{equation}\label{eq:expansion-in-red} \#\{(x,y)\in E_\textsc{r} \,:\; x\in S\,,\, y\in S^c\} \geq \tfrac12 b|S|\quad\mbox{ for every $k$-thick $S\subset V(\cG)$ with $\epsilon n \leq |S|\leq n/2$}\,.\end{equation}
\end{claim}
\begin{proof}
By the edge expansion assumption, the probability that a given $k$-thick subset $S$ as above --- denoting its size by $s$ --- violates the statement of the claim, is at most
\[ \binom{\lceil b s\rceil }{\lceil  b s/2\rceil } \hat\epsilon^{bs/2} \leq \left[(4+o(1))\hat\epsilon\right]^{bs/2}\,. \]
On the other hand, if $m=\lceil s/k\rceil$, then there are at most $2^{2s} \binom{n}{m} (ed)^{s}$ such subsets, as we may enumerate over all possible ways to write $s=s_1+\ldots+s_m$ for the subset sizes (some possibly empty; there are at most $m$ components since nonempty ones have size at least $k$), and then for each subset $S_i$ we first root it in some $v_i\in V(\cG)$ --- the total number of choices of these roots would be $\binom{n}m$ --- and then specify its spanning tree out of at most $(ed)^{s_i}$ options (as all degrees are at most $d$; see, e.g.,~\cite[Lemma 2]{BFM98}).

Hence, the probability that there exists some set $S$ violating the statement of the claim is at most
\[ \sum_{\epsilon n \leq s \leq n/2} \bigg[(1+o(1)) \left(\frac{e n}{s/k}\right)^{1/k+o(1)} 4ed (4\hat\epsilon)^{b/2}\bigg]^s \leq \sum_{s\geq \epsilon n} \left[ 12d (4\hat\epsilon)^{b/2}\right]^s \leq \sum_{s\geq \epsilon n} 2^{-s} = O(2^{-\epsilon n})\,,
\]
where we used that $k$ is large enough such that $4e(e\epsilon^{-1} k)^{1/k} <12$, as well as~\eqref{eq:epsilon-hat-small}.
\end{proof}
Using the above claim, we can produce a version of Claim~\ref{clm:menger} that will only consider $E_\textsc{r}$ for sprinkling.
\begin{claim}\label{clm:menger-red}
For every $b,d,\epsilon>0$ there are $k,c>0$ such that the following holds. If $\cG$ is a regular $(b,d)$-expander on $n$ vertices,  $E(\cG) = E_\textsc{b}\cup E_\textsc{r}$ satisfying~\eqref{eq:expansion-in-red}, $\mathcal{S}$ is a family of disjoint connected subsets of $\cG$, each of size at least $k$,  
and $H \sim \cG_{\epsilon}$ independently,
then with probability $1-\exp(-cn)$, every two subsets $\mathcal{S}_1,\mathcal{S}_2$ of $\mathcal{S}$, with at least $\epsilon n$ vertices in each, are connected by a path in $H_\textsc{r}$.
\end{claim}
\begin{proof}
Each of the subsets $\mathcal{S}_1$ and $\mathcal{S}_2$ addressed by the claim is by definition $k$-thick in $\cG$ (by the hypothesis on $\mathcal{S}$) and contains at least $\epsilon n$ vertices. Thus, for any such $\mathcal{S}_1$ and $\mathcal{S}_2$, if $(U,U^c)$ is a minimal cut separating these subsets of vertices in $\cG_\textsc{r}$ (i.e., $\mathcal{S}_1\subset U$ and $\mathcal{S}_2\subset U^c$ for $U$ minimizing the number of edges between $U,U^c$ in $E_\textsc{r}$) then w.l.o.g.\ $U$ is  also $k$-thick and $|U|\leq n/2$ (suppose $U$ is the smaller part of the cut; we proceed by moving vertices that are not connected to $\mathcal{S}_1$ from $U$ to $U^c$, noting that this would not increase the cut size). Hence,~\eqref{eq:expansion-in-red} implies that there are, with probability $1-O(2^{-\epsilon n})$, at least $\frac12\epsilon bn$ edges in this cut. By Menger's Theorem, as in the proof of Claim~\ref{clm:menger}, we thus conclude that~\eqref{eq:short-paths-in-G} holds for $A=\mathcal{S}_1$ and $B=\mathcal{S}_2$ with $b/2$ replacing $b$, and the rest of the argument in the proof of Claim~\ref{clm:menger} holds as before (with $b$ replaced by $b/2$).
\end{proof}
Following the same short proof of Corollary~\ref{cor:V1-after-sprinkling} with the single modification of using the above claim instead of Claim~\ref{clm:menger} now yields its following analog w.r.t.\ sprinkling only \textsc{red} edges.

\begin{corollary}\label{cor:V1-after-sprinkling-red}
For every $b,d,\epsilon>0$ there are $c,R>0$ so that the following holds. If $\cG$ is a regular $(b,d)$-expander on $n$ vertices with girth greater than $2R$, and  $E(\cG) = E_\textsc{b}\cup E_\textsc{r}$ satisfies~\eqref{eq:expansion-in-red}, then with probability $1-O(e^{-cn})$, there is a component $\cC$ of $G'_\textsc{r} \cup F'_\textsc{r}$ containing all but at most $2\epsilon n$ vertices of
\[\overline{V}_1(G'_\textsc{r}) := \left\{ y \in V(G) \,:\; \cA_{x,y}^{R}(G'_\textsc{r})\mbox{ holds for some }x\in V(G)\right\}\,.\]
\end{corollary}

Condition on the partition of $E(\cG)$ into $E_\textsc{b}\cup E_\textsc{r}$, and suppose that~\eqref{eq:expansion-in-red} holds (which occurs with probability $1-O(\exp(-cn))$ as per Claim~\ref{clm:G-red}). 
Observe that if $e=(x,y)\in \widehat{\mathcal{U}}$ then $x,y\in \overline{V}_1(G'_\textsc{r})$ thanks to the events $\cA^R_{y,x}(G'_\textsc{r}),\cA^R_{x,y}(G'_\textsc{r})$, respectively.
By the last corollary, with probability $1-O(\exp(-cn))$, all the edges of $\widehat{\mathcal{U}}$ up to at most $2d\epsilon n$ will thus belong to some component $\cC$ of $G'\cup F'_\textsc{r}$. Moreover, each such edge is \textsc{blue}, and its endpoints will be connected in $\cC$ by a \textsc{red} path (in $G'_\textsc{r} \cup F'_\textsc{r}$). In particular, with probability $1-O(\exp(-cn))$, all but at most $2d\epsilon n$ edges of $\widehat{\mathcal{U}}$ are in the 2-core of  $\cC$.

Note that, as argued above~\eqref{eq:E1-mean}, the girth assumption on $\cG$ implies that the events $\cA^R_{x,y}$ and $\cA^R_{y,x}$ are independent, and by their definition these are also independent of the event $\{(x,y)\in E_\textsc{b}\}$ (which occurs with probability $\hat\epsilon$). Therefore, our bounds on $\P(\cA^R_{x,y})$ yield
\[ \tfrac12 d  p' (1-q')^2 \hat\epsilon n \leq  \E|\widehat{\mathcal{U}}| \leq \tfrac12 d  p' (1-q'+\epsilon)^2 \hat\epsilon n \,,\]
and, as argued before, the standard concentration estimate via the edge-exposure martingale yields
\[ \P\left(\big||\widehat{\mathcal{U}}|-\E|\widehat{\mathcal{U}}|\big|\geq a\right) \leq \exp\left(-\frac{a^2}{4 d n (d-1)^{2R}}\right)\,,\]
so that, in particular, for some $c(\epsilon,d,R)>0$ and every sufficiently large $n$,
\[ \P\left(|\widehat{\mathcal{U}}|\geq \tfrac12 dp'(1-q')^2\hat\epsilon n - \epsilon n\right) = 1 - O(\exp(-c n))\,.\]
In conclusion, the aforementioned single (linear) component $\cC$ with most of the edges of $\widehat{\mathcal{U}}$ must w.h.p.\ be the largest component $\GC$ in light of our result in the previous section. 

Denoting the number of edges and \textsc{blue} edges in the 2-core of $\GC(G)$, respectively, by
\[ M := \left| E(\TC(G)) \right| 
\quad\mbox{ and }\quad 
\widehat M := \left| E(\TC(G)) \cap E_\textsc{b} \right|\,,\]
we conclude from the above analysis of $\widehat{\mathcal{U}}$ that, for some $c(\epsilon,d,R)>0$,
\[ \P\left(\widehat M \geq \tfrac12 d p'(1-q')^2 \hat\epsilon n - (2d+1)\epsilon n\right) = 1-O(\exp(-cn))\,.\] 
At the same time, since the partition $E(\cG) = E_\textsc{b}\cup E_\textsc{\red}$ was performed independently of $G\sim\cG_p$, we have that, conditional on $G$, the random variable $\widehat M$ is distributed as $\mathrm{Bin}(M,\hat\epsilon)$. In particular, 
\[ \P \left( \widehat M \leq M \hat\epsilon +\epsilon n \mid G\right) \geq 1-\exp\left[-(\epsilon n)^2/(2 M)\right] \geq 1-\exp\left[-(\epsilon^2/d) n\right] \]
by Hoeffding's inequality and the fact that $M\leq dn/2$ deterministically. Combining the last two inequalities, it follows that, for some $c(\epsilon,d,R)>0$,
\[ \P\left( M \geq \tfrac12 dp'(1-q')^2 n - (2d+2)(\epsilon /\hat\epsilon) n\right) = 1 - O(\exp(-cn))\,,
\]
and the fact that $(2d+2)/(\epsilon/\hat\epsilon) < \sqrt{\epsilon}$ by our definition of $\hat\epsilon$ in~\eqref{eq:epsilon-hat-def} 
(so that, with this probability, $M \geq \tfrac12 dp'(1-q')^2 n - \sqrt{\epsilon} n$, where as before, $\tfrac12 dp'(1-q')^2 \to\tfrac12 dp(1-q)^2$ as $\epsilon\to 0$) now establishes~\eqref{eq-E2-in-C2+} and thereby concludes the proof of Proposition~\eqref{prop:E2-V2-approx}.
\end{proof}

\begin{remark}\label{rem:2-core-bridges}
Note that our lower bound on $|\TC(G)|$ --- which was tight up to at most $\epsilon n$ vertices --- consisted of edges in the 2-core that lie on a cycle; thus, w.h.p.\ there are at most $\epsilon n$ bridges in $\TC(G)$.
\end{remark}

\begin{remark} \label{rem:sprinkling-red}
A subtler aspect of the lower bound on $|\TC(G)|$ is that counted edges that are \textsc{blue} and lie on cycles that, apart from this edge, are entirely \textsc{red}; up to an arbitrarily small error, this matched the correct number of edges in $\TC(G)$, where edges may lie on cycles that have mixed colors, thus the effect of such mixed cycles is negligible! It is important to note that this is \emph{not} the case for the number of such cycles (e.g., consider the case when the girth tends to infinity), but rather for the sake of determining whether a given edge is contained in such a cycle.\end{remark}

From Proposition~\ref{prop:E2-V2-approx}, Eq.~\eqref{eq:approx-E2}, and~\eqref{eq:E2-concentrate}--\eqref{eq:E2-size}, we deduce the required estimate on $|E(\TC)|$. 
Moreover, with probability $1-\exp(-cn)$, at most $2\epsilon n$ vertices in $G\sim\cG_p$ have a discrepancy between their degrees in $\TC(G)$ and in $E_2(G)$. As before, the statement of Theorem~\ref{thm:giant-2core-deg-dist} that $|D_k^*/n-\beta_k|<\epsilon$ follows from the fact that, for every $x$ and $k\geq 2$, the probability that $\#\{y : xy\in E_2(G)\}=k$ corresponds to $\beta_k$, as this occurs iff $x$ has exactly $k$ neighbors $y$ in $G$ such that $\{ |\cC_{y}(\TC\setminus\{x y\})|\geq R\}$.
(By~\eqref{eq-q-def}, the expression for $\beta_k$ in~\eqref{def:beta_k} equals
$\binom{d}{k}   p^k (1-q)^k \left(1-p+pq\right)^{d-k}$, i.e., $\P(\Bin(d,p(1-q))=k)$.)

\section{Second largest component}

\subsection{Proof of Theorem~\ref{mainthm:second-comp}}\label{subsec:second-comp}
Fix $d\geq 3$ and $R\geq 1$, and let $\frac1{d-1}<p<1$ and $0<\alpha<1$. 
We need the following result, which (although it
may be proved directly) follows immediately from our results in~\S\ref{sec:asymp}.
\begin{claim}\label{clm:expander-matching}
For every $b>0$, $d\geq 3$ and $\frac1{d-1}<p<1$ there exist some $c,\delta,\epsilon,R>0$ such that, if $\cG=(V,E)$ is a regular $(b,d)$-expander on $n$ vertices with girth at least $R$, then there exists a subset $M$ of at least $\delta n$ vertex-disjoint edges of $\cG$ such that the graph $\cG_0=(V,E\setminus M)$ satisfies that $G_0\sim (\cG_0)_p$ contains a connected component of size at least $\epsilon n$ with probability $1-O(e^{-c n})$.
\end{claim}
\begin{proof}
Using the notation in~\S\ref{sec:2-core}, let each edge in $\cG$ be \textsc{blue} independently with probability $\hat\epsilon$, as defined in~\eqref{eq:epsilon-hat-def} and \textsc{red} otherwise; denote the \textsc{red} and \textsc{blue} edges by $E_{\textsc{r}}$ and $E_{\textsc{b}}$, respectively.
Corollary~\ref{cor:V1-after-sprinkling-red} guarantees that, with probability $1-O(e^{-cn})$, the red graph $\cG_\textsc{r}=(V,E_{\textsc{r}})$ satisfies that $G_{\textsc{r}}\sim (\cG_{\textsc{r}})_p$ contains a component of size $\epsilon n$ for some $\epsilon>0$. Finally, $E_{\textsc{b}}$ is of size at least $(\hat\epsilon/2)dn/2 $ with probability  $1-O(e^{-cn})$, whence, in particular, it contains a matching $M$ of size at least $\hat\epsilon n/4$. 
\end{proof}

Our construction of a regular $(b,d)$-expander $\cG$ on $n$ vertices is as follows.
\begin{itemize}
	\item Fix some arbitrary $b_1>0$. Let $R_1$ be the maximum of $R$ and the girth requirement from Theorem~\ref{mainthm:giant-2core-sizes} w.r.t.\ $\epsilon=\theta_1/2$, and construct an expander $H_1$ on $n_1 \asymp n^\alpha$ vertices with girth at least $R_1$, where $\hat{n}_1 \asymp  n^{\alpha}$  vertices have degree $2$ and the rest have degree $d$, in the following way:
	As per Claim~\ref{clm:expander-matching}, let $H_1$ be obtained from a regular $(b_1,d)$-expander on $n_1 = \lfloor  n^{\alpha}\rfloor $ vertices with girth at least $R_1$ by subdividing each of the edges of $M$ (given by that claim) into a path of length 2 via a new degree-2 vertex (for a total of $\hat n_1 \asymp n^\alpha$ new degree-2 vertices).
	Let $\hat{V}_1$ denote the $\hat n_1$ vertices of degree $2$ in $H_1$.
	
	\item On each vertex $v\in \hat V_1$ in $H_1$: connect it via new edges to $d-2$ new $(d-1)$-ary trees $T_v$, each of depth \[ h=\left\lceil\frac12(1-\alpha)\log_{d-1} n \right\rceil\,.\] 
	\item Let $F$ be a $d$-regular graph on $m$ vertices, for some fixed $m$, with  girth at least $R$ --- e.g., for concreteness, the Erd\H{o}s--Sachs~\cite{ES63} graph --- and an arbitrary edge $(x,y)$ removed from it. Take
	\[ L = \left\lceil\frac{2+\alpha}{(1-\alpha)\log_{d-1}(1/p)} \right\rceil \,,\]
	and replace every edge $(a,b)$ of each $(d-1)$-tree $T_v$ by a path of $L$ copies of $F$ (with $\{(x_i,y_i)\}_{i=1}^L$ denoting their deleted edges), where adjacent copies have $y_i$ connected to $x_{i+1}$ by a new edge, the vertex  $a$ is connected by a new edge to $x_1$, and $b$ is connected by a new edge to~$y_L$. 
	
	Note that this did not modify any of the degrees in the original vertices of the $(d-1)$-ary tree $T_v$ (which are thus $d$ everywhere except for the degree-1 leaves). 

	\item Identify each of the $\hat n_2 \asymp n^{\frac12 (1+\alpha)}$ total leaves in all trees $T_v$ as above, in an arbitrary way, to the vertices of one final regular $b_2$-expander $H_2$ 
with $\hat{n}_2$ vertices of degree $d-1$ and $n_2$ vertices of degree $d$, whose girth is at least $R_1$, where $n_2\asymp n$ is such that altogether there are $n$ vertices.
\end{itemize}
It is easy to see that the $\cG$ is a regular $(b,d)$-expander for some fixed $b>0$, depending only on $b_1,b_2,d$, as every set of $s$ vertices must have at least $s/3$ vertices belong either to $H_1$, or to the interior vertices of the trees $\{ T_v: v\in H_1\}$, or to $H_2$, thus the required expansion can be inferred from the one within that corresponding expander (out of the three) by itself. It remains to examine $\cG_p$.

By Theorem~\ref{mainthm:giant-2core-sizes}, the size of the largest component $\GC$ in $(H_1)_p$ is of order $ n^{\alpha}$ (thanks to Claim~\ref{clm:expander-matching}) with probability at least $1-\exp(-c n^{\alpha})$. Each vertex $v$ of this component connects to $H_2$ in $\cG_p$ (i.e., there is a path in $(T_v)_p$ from the root to one of the leaves) with probability at most $(d-1)^h p^{h L}$; thus, the probability that $H_1$ is connected to $H_2$ in $\cG_p$ is at most
\[ \hat n_2 \, p^{h L} \asymp n^{(1+\alpha)/2} p^{h L} \leq n^{-1/2}=o(1)\,.\] Furthermore, the size of this augmented component in $\cG_p$ is stochastically dominated by the combined size of  $O(n^{\alpha})$ i.i.d.\ GW-trees with offspring distribution $\Bin(d-1,p^L)$ (multiplied by an extra factor of $L|V(F)|=O(1)$  due to the $L$ copies of the graph $F$). Since $p^L<d^{-2}$, in particular these GW-trees are subcritical, hence the size of this component is of order $n^{\alpha}$ w.h.p.

On the other hand, the entire graph has girth at least $R$ and is a $d$-regular expander, hence contains a linear component with probability at least $1-\exp(-c n)$ by Theorem~1.
\qed

\begin{remark} In the above construction, unlike the case of $\cG(n,p)$, some edges can have a polynomially small probability of belonging to the giant, even though the girth can be made arbitrarily large and in particular the local neighborhoods (to an arbitrary distance) look all alike.	
\end{remark}

\subsection{Proof of Theorem~\ref{mainthm:q-benjamini}}

The basic building block in the construction, which was already used for the construction in the proof of Theorem~\ref{mainthm:second-comp}, is the tree  $T(k,h,h^*,L^*)$, for integers $k\geq 2$ and $h,h^*,L^*\geq 1$, obtained by
\begin{enumerate}[(a)]
\item taking a $k$-ary tree with $h$ levels, and
\item subdividing all edges between parent nodes in the last $h^*$ levels (i.e., in every level $j\geq h - h^*$) and their children to a path of length $L^*$. 
\end{enumerate}
Let $0<p<1$, fix some sufficiently small $\epsilon>0$, and set
 \begin{equation}\label{eq:cor3-d-def}
 d := \lceil (1/p)^{2/\epsilon}\rceil\,, 	
 \end{equation}
and define
\[ h_n := \lfloor \log_d n\rfloor \,,\quad \alpha := \left\lceil \epsilon^{-1} \log_{1/p} d\right\rceil\,,\quad \beta:= \left\lceil 2\log_{1/p} d\right\rceil\,.\]
With these notations, our construction is the following graph $\cG$ on $|V(\cG)|\asymp n$ vertices.
\begin{enumerate}[(i)]
	\item Let $ T_1 = T(d, h_n, \lfloor\epsilon h_n\rfloor, \alpha) $. Denote its root by $v$ and its leaves by $\{u'_1,\ldots,u'_{N}\}$ (so $n/d< N \leq n$).

	\item Let 
$ T_2 =  T(d,h_n, \lfloor\epsilon h_n\rfloor, \beta)$
 and denote its leaves by $\{u''_1,\ldots,u''_{N}\}$.

\item Let $F$ be the graph formed by taking some (arbitrary) $d$-regular expander $F_0$ on the vertex set $\{w_1,\ldots,w_{N}\}$, and then subdividing each edge of $F_0$ into a path of length $\beta$.

\item Add the edges $(u'_i,w_i)$ and $(u''_i,w_i)$ for all $i=1,\ldots,N$.
\end{enumerate}
\begin{claim}\label{clm:cons3-diam(G)}
For large enough $n$, the graph 
$\cG$ above satisfies $\diam(\cG) \leq (1+6\epsilon) \log_{1/p} n$.
\end{claim}
\begin{proof}
First note that the height of $T_2$ is 
\[ \ell_2 := h_n + (\beta -1) \lfloor \epsilon h_n\rfloor \leq \log_d n + 2\epsilon \log_{1/p} n \leq \tfrac52\epsilon \log_{1/p} n\,,\]
where we used the fact that $\log_d n < (\epsilon/2)\log_{1/p}n$ by~\eqref{eq:cor3-d-def}. Therefore 
\[ \dist_{\cG}(x,y) \leq  R := 2\ell_2 + \beta+2 =5\epsilon\log_{1/p}n+O(1) \quad\mbox{ for every  $x,y\in V(F)\cup V(T_2)$}\,,\]
 where the additive $\beta+2$ accounts for possibly traversing from the vertex $x$ to the closest vertex $w_i$ to it and then to $v_i$ (at total distance at most $\beta/2+1$ from $x$), and similarly for $y$.
 
Next, consider $x\in V(T_1)$ vs.\ $y\in V(F)\cup V(T_2)$. If $u'_i$ is some descendent of $x$ in the tree $T_1$, then as established above, $\dist(y,u'_i) \leq R$, while $\dist(x,u'_i)$ is at most the height of $T_1$, given by
\[ \ell_1 := h_n + (\alpha-1)\lfloor\epsilon h_n\rfloor \leq \log_d n + \log_{1/p} n \leq (1+\epsilon/2) \log_{1/p} n\,,\]
and overall
\[ \dist(x,y) \leq \ell_1 + R \quad\mbox{ for every  $x\in V(T_1)$ and $y\in V(F)\cup V(T_2)$}\,.\]
Finally, consider $x,y\in V(T_1)$, let $u'_i$ and $u'_j$ be some descendants of $x,y$, respectively, and write $d_x = \dist(x,u'_i)$ and $d_y = \dist(y,u'_j)$. If $d_x + d_y \leq \ell_1$, then  $\dist(x,y) \leq d_x+d_y + \dist(u'_i,u'_j)\leq \ell_1+ R$, and otherwise 
$ \dist(x,y)\leq \dist(v,x)+\dist(v,y) \leq (\ell_1 - d_x)+(\ell_2-d_y) < \ell_1$, thus overall,
\[ \dist(x,y) \leq \ell_1 + R \quad\mbox{ for every  $x\in V(T_1)$ and $y\in V(\cG)$}\,.\]
Plugging in the values of $\ell_1$ and $R$ we see that
$ \diam(\cG) \leq \ell_1+R \leq (1+6\epsilon)\log_{1/p}n$ 
for all sufficiently large $n$ (absorbing the additive $O(1)$-term in $R$ via the increased pre-factor of $\epsilon \log_{1/p}n$), as required.
\end{proof}

\begin{claim}\label{clm:cons3-no-giant}
W.h.p., the random graph $G\sim \cG_p$ satisfies $|\cC_1(G)| \leq n^{1-\epsilon/3}$.
\end{claim}
\begin{proof}
Let $U_1$ denote the vertices $x\in V(\cG)$ which are at distance at most $(1-\epsilon/2 + \alpha\epsilon/2)h_n$ from $v$, the root of $T_1$ (these are the vertices of the subtree of $T_1$ in the first $(1-\epsilon/2)h_n$ levels of the tree before the subdivision, as well as the new vertices in the subdivided edges between them). Similarly, let $U_2$ denote the vertices $x\in V(\cG)$ whose distance from the root of $T_2$ is at most $(1-\epsilon/2 + \beta\epsilon/2)h_n$, and let $U_3=V(\cG)\setminus(U_1\cup U_2)$. 
By construction, $ |U_i| = O(d^{(1-\epsilon/2)h_n}) = O(n^{1-\epsilon/2})$ for $i=1,2$.

At the same time, if $x\in U_3$ then exploring the neighborhood of $x$ in $\cG$ via Breadth-First-Search reveals edges that were subdivided (in either $F$, $T_1$ or $T_2$) into paths of length at least $\beta$, to within depth (prior to the subdivision) of at least $\ell := \lceil (\epsilon/2) h_n\rceil $. In particular, this exploration process in $G\sim\cG_p$ is stochastically dominated by a Galton--Watson tree with offspring variable $\Bin(d,p^\beta)$. Thus, if $\zeta_i$ are i.i.d.\ such random variables, and $\Lambda$ is the total progeny of this Galton--Watson tree, then
 \begin{align*}
 \P(\Lambda > k) & \leq \P\left(\sum_{i=1}^k(\zeta_i-1)\geq 0\right) \leq \P\left(\Bin(k d,p^\beta)\geq k\right)\leq \binom{kd}{k}p^{\beta k} \leq \left(ed p^\beta\right)^k\,,
 \end{align*}
 which, for $k=\ell$ and using $p^\beta<d^{-2}$, is at most $(d/e)^{-\ell} = O\big(n^{-\frac25\epsilon}\big)$ provided that $\epsilon$ is sufficiently small.
So, if 
$ U'_3 = \{ x\in U_3 \,:\; |\cC_x(G)| > \ell\}$, then $\E |U'_3| = O(n^{1-\frac25\epsilon})$.
As $|\cC_1(G)| \leq \max\{ \ell, |U_1|+|U_2|+|U'_3|\}$, we find that 
$ \E |\cC_1(G)|  = O(n^{1-\frac25\epsilon})$,
which implies the statement of the claim.
\end{proof}
\begin{claim}\label{clm:cons3-diam(C)}
For large enough $n$, the random graph $G\sim \cG_p$ satisfies
\[ \P\left(\diam(\cC_v(G)) > (1-\epsilon) \log_{1/p}n\right) \geq 1-\epsilon\,.\]
\end{claim}
\begin{proof}
Let $A$ denote the vertices of $T_1$ at distance $h_n^* = h_n-\lfloor\epsilon h_n\rfloor$ from its root $v$, and let $B$ be the set of vertices whose distance from $v$, prior to the subdivision of the edges of the tree $T_1$, was $h_n^* + \ell$, where $\ell = \lfloor(\epsilon-2\epsilon^2) h_n\rfloor$.
The probability that percolation on the subtree of $T_1$ rooted at some $x\in A$ survives to intersect $B$ is at least
\[  p^{\alpha\ell} \leq d^{-\epsilon^{-1} (\epsilon -2\epsilon^2) h_n} = n^{-1+2\epsilon}\,.\] 

Recall that, if $Z_t$ is the size of generation $t$ in a Galton--Watson tree with offspring variable $\Bin(d,p)$, such that $m=dp>1$ and its extinction probability is $0<q<1$, then $Z_t m^{-t}$ converges a.s.\ as $t\to\infty$ to a random variable $W$ which, except for a mass of $1-q$ at $W=0$, has an absolutely continuous distribution on $(0,\infty)$ (see, e.g.,~\cite[{\S}I.12]{AthreyaNey04}). In particular, there exists $\delta>0$ such that $\P(Z_t \geq \delta m^t) > 1-q-\epsilon/3$ for every sufficiently large $t$. Further, since $q$ is monotone decreasing in $d$, our choice of $d$ in~\eqref{eq:cor3-d-def} readily implies (recall~\eqref{eq-q-def}) that $q<\epsilon/3$ provided $\epsilon$ is small enough, thus overall  $\P(Z_t \geq \delta m^t) > 1-\frac23\epsilon$ for every large enough $t$. Specialized to our setting, we take $t=h_n^*$, and noting that
\[ m^{h_n^*} \geq d^{-1} n^{1-\epsilon} p^{(1-\epsilon)\log_d n} \geq d^{-1} n^{1-\frac32 \epsilon + \frac12\epsilon^2}\]
(using that $\log (1/p) < \frac12\epsilon\log d$ by~\eqref{eq:cor3-d-def}), we infer that the following bound on the size of $A'$, the set of all vertices $x\in A$ such that there exists a path from $v$ to $x$ in $G$: for every sufficiently large $n$,
\[ \P\left(|A'| \geq n^{1-\frac32\epsilon} \right) \geq 1-\tfrac23\epsilon\,.\]
On the event $|A'|\geq n^{1-\frac32\epsilon}$,  the size of the set $B'$ of vertices $y\in B$ that are connected to $v$ in $G$ stochastically dominates a random variable $Z\sim \Bin(n^{1-\frac32\epsilon}, n^{-1+2\epsilon})$ by our bound on the event that the subtree of $x\in A$ survives the percolation to intersect $B$. Since $Z>0$ w.h.p.\ (being concentrated around $n^{\epsilon/2}$), it then follows that $\P(B'\neq \emptyset) \geq 1-\tfrac23\epsilon-o(1)$, whereas every vertex $y\in B'\cap\cC_v(G)$ would imply that 
\[  \diam(\cC_v(G))\geq \dist(v,y) \geq h_n^* + \alpha \ell > (1-2\epsilon)\log_{1/p}n\,.\qedhere \]
\end{proof}
As argued in the proof of Theorem~\ref{mainthm:second-comp}, the graph obtained by connecting the expander $F$ to $T_i$ is clearly also an expander, hence (iterating this) so is the entire graph $\cG$.
The combination of Claims~\ref{clm:cons3-diam(G)},~\ref{clm:cons3-no-giant} and~\ref{clm:cons3-diam(C)} thus concludes the proof of Theorem~\ref{mainthm:q-benjamini}. \qed

\section{Separators, long paths and complete minors}\label{sec:separators_long_paths}

In this section we research further typical properties of a random subgraph of a high girth constant degree expander. The properties are: non-existence of a small separator, existence of a linearly long path, and existence of a large complete minor. The results obtained are ``soft", so to say, and lack the precision of the conclusions of Theorem~\ref{mainthm:giant-2core-sizes}; still, we believe they are of interest and complement nicely the more accurate results.
We start with arguing that a random subgraph of a high girth constant degree expander typically has all its separators linear in $n$. Given a graph $G =(V,E)$ on $n$ vertices, a vertex set $S\subset V$ is called a {\em separator} if there is a partition $V = A\cup B\cup S$ of the vertex set of $G$ such that $G$ has no edges between $A$ and $B$, and
$|A|, |B|\le 2n/3$. Separators serve to measure quantitatively the connectivity of large vertex sets in
graphs; the fact that all separators in $G$ are large indicates that it is costly to break $G$ into large
pieces not connected by any edge. 

\begin{theorem}\label{thm-separators}
Fix $d\geq 3$ and $\frac{1}{d-1}<p<1$.
For every $b>0$ there exist some $\delta,c_1,R>0$ such that, if $\cG$ is a regular $(b,d)$-expander on $n$ vertices with girth exceeding $R$, then
with probability $1-O(e^{-c_1n})$, the largest connected component $\GC$ of the random graph $G\sim \cG_p$ has no separator of size at most $\delta n$.
\end{theorem}

\begin{proof}
We first describe the idea of the proof. We will argue that if $G\sim \cG_p$ has a small separator $S$, then by deleting all edges touching $S$ we get a graph $G'\sim \cG_{p'}$ without a connected component of size as large as dictated by Theorem~\ref{mainthm:giant-2core-sizes}, for some $\tfrac1{d-1}<p'<p$. Since deleting a relatively small number of edges touching $S$ incurs a relatively small penalty when going from $G$ to $G'$, we will thus conclude that the probability of $G$ to have a small separator must be exponentially small to begin with.

Now we provide a full proof, implementing rigorously the above described outline. 
For $\tfrac1{d-1}<p'<p$, to be chosen momentarily, denote
\[
\theta_1:=\theta_1(p)\,,\quad \theta_1':=\theta_1(p')\,,
\]
where we apply~\eqref{eq-theta-def} to define $\theta_1,\theta_1'$. Choose $p'$ so $\theta_1'=\frac{5}{6}\theta_1$ (this is possible by the continuity of $\theta_1(p)$),
and let
\[
\rho=\frac{p'}{p}\,.
\]
Notice that a random graph $G'\sim \cG_{p'}$ can be obtained first by drawing a random graph $G\sim \cG_p$, and then by retaining every edge of $G$ with probability $\rho$ independently.

Let $\cA$ be the following event addressing $G\sim \cG_p$ for some (small enough) $\delta>0$ to be set later:
\[ \cA=\big\{\mbox{ $\exists\,S\subset [n]$\,:\; $|S|\le \delta n$ and $|\cC_1(G\setminus S)|\leq \frac{11}{15}\theta_1n$ }\big\}\,,\]
and let $\cB$ be the following event addressing $G'\sim \cG_{p'}$:
\[ \cB=\big\{ |V(\GC(G'))|\le \tfrac{11}{15}\theta_1n\big\}\,.\]
 Suppose $G$ satisfies $\cA$, and choose $S$ as in the definition of $\cA$. There are at most $d|S|\le \delta dn$ edges touching $S$ in $G$. The probability to erase all these edges when going from $G$ to $G'$ is at most $(1-\rho)^{\delta dn}$. However, if none of these edges belongs in $G'$, then $|\cC_1(G')|\leq \frac{11}{15}\theta_1n=\frac{22}{25}\theta_1'n$.  It thus follows that 
 \[ \P(\cB)\ge \P(\cA)\, (1-\rho)^{\delta dn}\,.\]
 On the other hand, by Theorem~\ref{mainthm:giant-2core-sizes}, the order of $\cC_1(G')$ is very close to $\theta_1'n$ with probability exponentially close to 1, making the event $\cB$ exponentially unlikely. Specifically, applying Theorem~\ref{mainthm:giant-2core-sizes} in $\cG_{p'}$
 with $\epsilon=\frac{3}{25}\theta_1'$, we have $\P(\cB)\le Ce^{-cn}$ for $C,c>0$. Hence,
$$
\P (\cA)\le Ce^{-cn}\,(1-\rho)^{-\delta dn}\le e^{-c'n}\,,
$$
for $\delta>0$ small enough (as a function of other parameters).

Finally, observe that if $G\sim \cG_p$ does not satisfy $\cA$, then for every subset $S$ of at most $\delta n$ vertices, the  graph $G-S$ has some connected component of size more than $\frac{11}{15}\theta_1n$. Invoking Theorem~\ref{mainthm:giant-2core-sizes} once again, this time in $\cG_{p}$ with $\epsilon=\frac{1}{10}\theta_1$, we derive that with probability exponentially close to 1, the random graph $G\sim \cG_p$ does not satisfy $\cA$, and its largest connected component satisfies $|\GC(G)| \leq \frac{11}{10}\theta_1n$. Therefore, in this case, recalling the definition of a separator, we derive that no subset $S$ of at most $\delta n$ vertices can be a separator in $\GC(G)$.
\end{proof}

As we have indicated, non-existence of small separators is a key fact in deriving other typical properties of (the giant component) of a percolated high girth expander in the super-critical regime. They are given in the following two corollaries.

\begin{corollary}\label{cor-long_path}
Fix $d\geq 3$ and $\frac{1}{d-1}<p<1$.
For every $b>0$ there exist some $\delta,c,R>0$ such that, if $\cG$ is a regular $(b,d)$-expander on $n$ vertices with girth exceeding $R$, then
with probability at least $1-\exp(-cn)$, the random graph $G\sim \cG_p$ contains a path of length at least $\delta n$.
\end{corollary}

\begin{proof} We need the following result (see, e.g.~\cite[Proposition~2.1]{Krivelev16}, including a simple proof):

\begin{lemma}\label{lemma-long_path}
Let $k,l$ be positive integers. Assume $G=(V,E)$ is a graph on more than $k$
vertices, where every $A\subset V$ of size $|A|=k$ has at least $l$ neighbors in $V\setminus A$. Then $G$ contains a
path of length~$l$.
\end{lemma}

According to Theorems~\ref{mainthm:giant-2core-sizes} and~\ref{thm-separators}, with probability exponentially close to 1, the giant component  $\GC$ of $G\sim \cG_p$ is of size close to $\theta_1n$ and has no separator of size at most $\delta n$; we can assume $\delta\ll \theta_1$. Let $A$ be a subset of $\GC$ of cardinality $|A|=|V(\GC)|/3$. Then by the definition of a separator, $A$ has at least $\delta n$ neighbors outside. It thus follows by Lemma~\ref{lemma-long_path} that $\GC$ contains a path of length at least $\delta n$.
\end{proof}

As it frequently happens in random graphs, getting a linearly long cycle from a linearly long path is pretty easy, here is a brief sketch of the argument. Choose $\frac{1}{d-1}<p'<p<1$ and argue that a random graph $G\sim \cG_{p'}$ has w.h.p. a path $P$ of length $\delta'n$, for some constant $\delta'>0$. Let $P_1,P_2$ be the first and the last thirds of $P$, respectively. Then by applying Menger's Theorem to the base graph $\cG$, we derive that it contains linearly many constant length paths between $V(P_1)$ and $V(P_2)$. With high probability at least one of these short paths survives sprinkling (taking us from $\cG_{p'}$ to  $\cG_p$); its union with $P$ contains a linearly long cycle.

We now discuss embedding complete minors in percolated expanders. Let $G=(V,E)$, and let $t>0$ be an integer.
We say that $G$ contains a {\em minor of the complete graph} $K_t$ if there is a
collection $(V_1,\ldots, V_t)$ of pairwise disjoint vertex subsets in $V$ such that each $V_i$ spans a connected
subgraph in $G$, and in addition $G$ has an edge between every pair of subsets $V_i,V_j$. Observe trivially that if $G$ contains a minor of $K_t$,  then $|E(G)|\ge \binom{t}{2}$; this trivial bound provides an obvious but meaningful benchmark for minor embedding statements.

Kawarabayshi and Reed proved in~\cite{KR10} that a graph $G$ on $n$ vertices has a minor of the complete graph $K_h$ or a separator of order $O(h\sqrt{n})$. Since by Theorem~\ref{thm-separators} the random graph $G\sim \cG_{p}$ has all separators of size at least $\delta n$ in its giant component $\GC$  (whose size is much larger than $\delta n$ by Theorem~\ref{mainthm:giant-2core-sizes}) with probability exponentially close to 1, we conclude:

\begin{corollary}\label{cor-complete_minor}
Fix $d\geq 3$ and $\frac{1}{d-1}<p<1$.
For every $b>0$ there exist some $\delta,c,R>0$ such that, if $\cG$ is a regular $(b,d)$-expander on $n$ vertices with girth exceeding $R$, then with probability $1-O(e^{-cn})$, the random graph $G\sim \cG_p$ contains a minor of $K_{\delta\sqrt{n}}$.
\end{corollary}

The order of magnitude for the maximal $h$ such that $G$ contains a minor of $K_h$  is obviously optimal, as the base graph $\cG$ has only linearly many edges to begin with.

\subsubsection*{Acknowledgement.} This work was initiated while the authors were visiting the Theory Group of Microsoft Research in Redmond, and they are grateful to the group for its hospitality.

\bibliographystyle{abbrv}
\bibliography{biblio}

\begin{thebibliography}{10}

\bibitem{ABS04}
N.~Alon, I.~Benjamini, and A.~Stacey.
\newblock Percolation on finite graphs and isoperimetric inequalities.
\newblock {\em Ann. Probab.}, 32(3A):1727--1745, 2004.

\bibitem{AthreyaNey04}
K.~B. Athreya and P.~E. Ney.
\newblock {\em Branching processes}.
\newblock Dover Publications, Inc., Mineola, NY, 2004.
\newblock Reprint of the 1972 original [Springer, New York].

\bibitem{Benjamini13}
I.~Benjamini.
\newblock Euclidean vs. graph metric.
\newblock In {\em Erd{\H{o}}s {C}entennial}, volume~25 of {\em Bolyai Soc.
  Math. Stud.}, pages 35--57. J\'anos Bolyai Math. Soc., Budapest, 2013.

\bibitem{BNP11}
I.~Benjamini, A.~Nachmias, and Y.~Peres.
\newblock Is the critical percolation probability local?
\newblock {\em Probab. Theory Related Fields}, 149(1-2):261--269, 2011.

\bibitem{BFM98}
A.~Beveridge, A.~Frieze, and C.~McDiarmid.
\newblock Random minimum length spanning trees in regular graphs.
\newblock {\em Combinatorica}, 18(3):311--333, 1998.

\bibitem{Bollobas01}
B.~Bollob{\'a}s.
\newblock {\em Random graphs}, volume~73 of {\em Cambridge Studies in Advanced
  Mathematics}.
\newblock Cambridge University Press, Cambridge, second edition, 2001.

\bibitem{DKLP11}
J.~Ding, J.~H. Kim, E.~Lubetzky, and Y.~Peres.
\newblock Anatomy of a young giant component in the random graph.
\newblock {\em Random Structures Algorithms}, 39(2):139--178, 2011.

\bibitem{DLP14}
J.~Ding, E.~Lubetzky, and Y.~Peres.
\newblock Anatomy of the giant component: the strictly supercritical regime.
\newblock {\em European J. Combin.}, 35:155--168, 2014.

\bibitem{ES63}
P.~Erd\H{o}s and H.~Sachs.
\newblock Regul\"are {G}raphen gegebener {T}aillenweite mit minimaler
  {K}notenzahl.
\newblock {\em Wiss. Z. Martin-Luther-Univ. Halle-Wittenberg Math.-Natur.
  Reihe}, 12:251--257, 1963.

\bibitem{JLR00}
S.~Janson, T.~{\L}uczak, and A.~Ruci\`nski.
\newblock {\em Random graphs}.
\newblock Wiley-Interscience Series in Discrete Mathematics and Optimization.
  Wiley-Interscience, New York, 2000.

\bibitem{KR10}
K.-i. Kawarabayashi and B.~Reed.
\newblock A separator theorem in minor-closed classes.
\newblock In {\em 2010 {IEEE} 51st {A}nnual {S}ymposium on {F}oundations of
  {C}omputer {S}cience---{FOCS} 2010}, pages 153--162. IEEE Computer Soc., Los
  Alamitos, CA, 2010.

\bibitem{Krivelev16}
M.~Krivelevich.
\newblock Long paths and hamiltonicity in random graphs.
\newblock In N.~Fountoulakis and D.~Hefetz, editors, {\em Random Graphs,
  Geometry and Asymptotic Structure}, London Mathematical Society Student
  Texts, pages 4--27. Cambridge University Press, 2016.

\bibitem{NP10}
A.~Nachmias and Y.~Peres.
\newblock Critical percolation on random regular graphs.
\newblock {\em Random Structures Algorithms}, 36(2):111--148, 2010.

\bibitem{Pittel08}
B.~Pittel.
\newblock Edge percolation on a random regular graph of low degree.
\newblock {\em Ann. Probab.}, 36(4):1359--1389, 2008.

\bibitem{vdHofstad}
R.~van~der Hofstad.
\newblock {\em Random graphs and complex networks. {V}ol. 1}.
\newblock Cambridge Series in Statistical and Probabilistic Mathematics, [43].
  Cambridge University Press, Cambridge, 2017.

\end{thebibliography}

\end{document}